\documentclass[10pt,draft,reqno]{amsart}
     \makeatletter
     \def\section{\@startsection{section}{1}%
     \z@{.7\linespacing\@plus\linespacing}{.5\linespacing}%
     {\bfseries
     \centering
     }}
     \def\@secnumfont{\bfseries}
     \makeatother
\newtheorem{theorem}{Theorem}[section]
\newtheorem{lemma}[theorem]{Lemma}

\theoremstyle{definition}
\newtheorem{definition}[theorem]{Definition}
\usepackage{enumerate}
\usepackage{xcolor}

\theoremstyle{remark}

\numberwithin{equation}{section}
\begin{document}

\title[Large Deviations for a Class of Semilinear SPDEs]{Large Deviations for a Class of Semilinear Stochastic Partial Differential Equations}

\author{Mohammud Foondun}
\address{ Mohammud Foondun: School of Mathematics, Loughborough University, Leicestershire, UK, LE11 3TU}
\email{m.i.foondun@lboro.ac.uk}

\author{Leila Setayeshgar}
\address{Leila Setayeshgar: Department of Mathematics and Computer Science, Providence College, Providence, RI 02918 , USA}
\email{lsetayes@providence.edu}




\begin{abstract}
We prove the large deviations principle (LDP) for the law of the solutions to a class of semilinear stochastic partial differential equations driven by multiplicative noise.  Our proof is based on the weak convergence approach and significantly improves earlier methods.
\end{abstract}

\maketitle

\section{Introduction }

We consider a family of semilinear stochastic partial differential equations, 
 
\begin{align}
\label{E:1}
\frac{\partial U^{\epsilon}}{\partial t}(t,x) &=\nonumber \frac{\partial^2 U^{\epsilon}}{\partial x^2}(t,x) +\sqrt{\epsilon}\sigma(t, x, U^\epsilon(t,x))\frac{\partial^2W}{\partial t \partial x}(t,x) \\& +\frac{\partial}{\partial x}g(t,x,U^{\epsilon}(t,x))+ f(t,x,U^\epsilon(t,x)),
\end{align}

\noindent with $U^{\epsilon}(t,0) = U^{\epsilon}(t,1) = 0$ for $t\in [0,T]$, and initial condition $U^{\epsilon}(0,x) = \eta(x)\in L^2([0,1])$.  $W(t,x)$ denotes the Brownian sheet \cite{BD} on a filtered a probability space, $(\Omega, {\mathcal F}, \{{\mathcal F}_t\}, P)$.  The functions $f=f(t,x,r)$,  $g=g(t,x,r)$,  $\sigma=\sigma(t,x,r)$ are Borel functions of $(t,x,r)\in \mathbb R_{+}\times[0,1]\times \mathbb R$.  Linear growth on $f$, and quadratic growth on $g$ are assumed.  Therefore, our family of semilinear equations contains, as special cases, both the stochastic Burgers' equation, and the stochastic reaction-diffusion equation. The existence and uniqueness to Eq. (\ref{E:1}) has been studied by  Gy\"ongy \cite{G} (1998), where the existence and uniqueness results of Bertini et al. \cite{BCJ} (1994), Da Prato et al. \cite{DDT} (1995), and Da Prato and Gatarek  \cite{DG} (1995), have been generalized.
Our aim is to prove the large deviation principle (LDP) for the law of the solutions to  Eq. (\ref{E:1}) by employing the weak convergence approach.  Our result generalizes the LDP for the stochastic Burgers' equation studied by Setayeshgar \cite{LS} (2014).  We state the precise statement of the large deviation principle below.

\begin{definition}[Large Deviation Principle]
  Let $I : \mathcal E \to [0,\infty]$ be a rate function on a Polish space $\mathcal E$. This means that for each $M<\infty$, the level set $\{x\in \mathcal E: I(x)\leq M\}$ is compact in $\mathcal E$.
  The sequence of random variables $\{X^\epsilon\}$ satisfies the large deviation principle on $\mathcal E$ with rate function $I$ if
  \begin{enumerate}
    \item For each closed subset $F\subset \mathcal E$,
    \[ \limsup_{\epsilon\rightarrow0} \, \epsilon \log P(X^\epsilon \in F) \leq - \inf_{x\in F} I(x). \]
    \item  For each open subset $G\subset \mathcal E$,
      \[ \liminf_{\epsilon\rightarrow0} \, \epsilon \log P(X^\epsilon \in G) \geq - \inf_{x\in G} I(x). \]  
  \end{enumerate}
\end{definition}

The Freidlin-Wentzell theory \cite{FW}, describes the {{asymptotic behavior}} of probabilities of the large deviations of the law of the solutions to a family of small noise finite dimensional SDEs, away from its law of large number limit. Here, we deal with the case where the driving Brownian motion is {{infinite dimensional}}.  In \cite{BDM},  Budhiraja et al. (2008) use certain {{variational representations}} for infinite dimensional Brownian motions \cite{BD} (originating from the work of Bou\'e and Dupuis \cite{BoD} (1998)) and demonstrate that, these representations provide a framework for proving large deviations for a variety of infinite dimensional systems, such as stochastic partial differentials equations.  One of the advantages of their method is that the technical exponential probability estimates usually used in proofs based on approximations, are no longer needed; instead, one is required to prove certain qualitative properties of the SPDE under study.

The following is the main contribution of this paper which establishes the large deviation principle for the law of the solutions to Eq. (\ref{E:1}).

\begin{theorem}[Main Theorem]\label{LDP}

\noindent The processes $\{U^{\epsilon}(t): t\in[0,T]\}$ satisfy the large deviation principle on $C\big([0,T]; L^2([0,1])\big)$ with rate function $I_{\eta}$ given by (\ref{ratefunction}).
\end{theorem}

The precise definition of the rate function is deferred to section 3.  Large deviations principle for Eq. (\ref{E:1}) has been studied by C. Cardon Weber \cite{C} (1997), using the classical approach.  The proof that we offer is totally different and is based on the weak convergence approach.  In this approach one proves the Laplace principle which is equivalent to the large deviation principle for Polish space random elements \cite{DE}. 

\begin{definition}[Laplace Principle]
 The sequence $\{X^n, n\in \mathbb N\}$ on a Polish space $\mathcal E$ is said to satisfy the Laplace principle with rate function $I$ if for all bounded continuous functions mapping $\mathcal E$ into $\mathbb R$

$$
\mbox{lim}_{n\rightarrow \infty}\frac{1}{n}{\mbox{log}} E\{\mbox{exp}[-nh(X^n)]\} = -\mbox{inf}_{x\in {\mathcal E}}\{h(x) + I(x)\}.
$$

\end{definition}

In the weak convergence approach which is suitable for the evaluation of integrals appearing in the Laplace principle, the integrals are associated to a variational representation through a family of minimal cost functions.  The asymptotic behavior of these minimal cost functions are in turn determined by the weak convergence approach \cite {DE}.  Finally, we note that compared to the proof of C. Cardon Weber \cite{C} (1997), the conditions require less technicalities. For instance, the time discretizations required in proving the regularity of the skeleton are avoided, and exponential inequalities for the stochastic integral in H\"older norms are no longer needed.  These are usually the most difficult parts of large deviations analysis based on the standard approximation method.  

In our proof based on the weak convergence approach, one only needs to establish the well-posedness of the controlled process, and its convergence to the limiting equation.  This results in a shorter, and more straightforward proof than that of C. Cardon Weber \cite{C} (1997).

We now give an outline of the paper. In Section 2, we state some assumptions and give some preliminary information.  The existence and uniqueness results for the family of semilinear SPDEs is also stated in this section.  In Section 3, we state the large deviations theorem due to Budhiraja et al. (\cite[Theorem 7]{BDM}) which we exploit.  We subsequently introduce the controlled and limiting equations, and establish their existence and uniqueness. Section 4 is devoted to the proof of the main theorem.  As noted before, establishing the large deviations principle hinges on proving the tightness and convergence properties of the controlled process.  This is carried out in Theorem 4.2.  
Unless otherwise noted, we adopt the following notation throughout the paper:  The notation $`` \doteq"$ means by definition.  $C$ denotes a  {{free}} constant which may take on different values, and depend upon other parameters. We use the notation $|h|_p$ to denote the $L^p([0,1])$-norm of a function $h$ defined on $[0,1]$.

\noindent\section{Main Assumptions and Preliminaries}

\noindent In this Section we introduce some assumptions and preliminaries which are needed for the formulation of the problem.  The functions $f=f(t,x,r)$,  $g=g(t,x,r)$,  $\sigma=\sigma(t,x,r)$ are Borel functions of $(t,x,r)\in \mathbb R_{+}\times[0,1]\times \mathbb R$ and have the following assumptions:

\begin{enumerate}[(H1)]

\item [(H1)] There exists a constant $K>0$ such that for all $(t,x,r)\in[0,T]\times[0,1]\times \mathbb R$ we have  $\sup_{t\in[0,T]}\sup_{x\in[0,1]}|f(t,x,r)|\leq K(1+|r|) $.

\item  [(H2)] The function $g$ is of the $g(t,x,r)=g_1(t,x,r)+g_2(t,r),$ where $g_1$ and $g_2$ are Borel functions satisfying
\begin{align*}
|g_1(t,x,r)|\leq K(1+|r|)\quad\text{and}\quad |g_2(t,r)|\leq K(1+|r|^2).
\end{align*}

\item [(H3)] $\sigma$ is bounded and for every $T\geq 0$ there exists a constant $L$ such that  for $(t,x,p,q)\in[0,T]\times[0,1]\times {\mathbb R}^2$ we have $|\sigma(t,x,p) -\sigma(t,x,q)|\leq L|p-q|.$
Furthermore, $f$ and $g$ are locally Lipschitz with linearly growing Lipschitz constant, i.e.,
\begin{align*}
|f(t,x,p) -f(t,x,q)|&\leq L(1 +|p|+|q|)|p-q|\\
|g(t,x,p)-g(t,x,q)|&\leq L(1 +|p|+|q|)|p-q|.
\end{align*}
\end{enumerate}
\begin{definition}[Mild Solution]\label{definition} \rm{A random field $U^{\epsilon} \doteq \{U^{\epsilon}(t,x) : t\in[0,T], x\in[0,1]\}$ is called a mild solution of (\ref{E:1}) with initial condition $\eta$ if $(t,x)\rightarrow U^{\epsilon}(t,x)$ is continuous a.s., and $U^{\epsilon}(t,x)$ is $\{{\mathcal F}_t\}$-measurable for any $t\in[0,T]$, and $x\in[0,1]$, and if

\begin{align*}
\label{E:Semilinear_Solution_definition}
&U^{\epsilon}(t,x)= \int_0^1 G_t(x,y)\eta(y)dy + \sqrt{\epsilon} \int_0^t\int_0^1 G_{t-s}(x,y) \sigma(s, U(s))(y)W(dy, ds)\\& -\int_0^t\int_0^1\partial_yG_{t-s}(x,y)g(s, U(s))(y)dyds+\int_0^t\int_0^1G_{t-s}(x,y)f(s, U(s))(y)dyds.
\end{align*}}
\end{definition}

\noindent The function $G_t(.,.)$ is the Green kernel associated with the heat operator $\partial/\partial t - \partial ^2/\partial^2x$ with Dirichlet's boundary conditions. The following result of Gy\"ongy (\cite[Theorem 2.1]{G}) asserts the existence and uniqueness of a solution to (\ref{E:1}).

\begin{theorem}[Existence and Uniqueness of Solution Mapping]\label{existence}

For any filtered probability space $(\Omega, {\mathcal F}, P, \{{\mathcal F}_t\})$, with a Brownian sheet defined as before, and $\eta\in L^p[0,1]$, $p\geq 2$ there exists a measurable function

$$
\xi^{\epsilon}: L^2([0,1]) \times C([0,T] \times[0,1]; \mathbb R) \rightarrow C\big([0,T] ; L^p([0,1])\big),
$$

\noindent such that $U^{\epsilon} \doteq \xi^{\epsilon}(\eta, \sqrt{\epsilon}W)$, (with $\eta$ denoting the initial condition) is the unique, mild solution of (\ref{E:1}).

\end{theorem}

We now state some estimates on the Dirichlet heat kernel. The proofs are well known and are omitted.

\subsection{Estimates on the Heat Kernel}
 
\noindent There exist positive constants $K, a, b, d$ such that for all $0\leq s<t\leq T$, and $x, y\in[0,1]$.
 
 \begin{enumerate}
 
 \item $|G(s,t;x,y)|\leq K\frac{1}{|t-s|}\exp\big (-a\frac{|x-y|^2}{t-s}\big)$,\\
 
 \item $|\frac{\partial}{\partial x} G(s,t;x,y)|\leq K\frac{1}{|t-s|^{3/2}}\exp\big (-b\frac{|x-y|^2}{t-s}\big)$,\\
 
 \item $|\frac{\partial}{\partial t} G(s,t;x,y)|\leq K\frac{1}{|t-s|^2}\exp\big (-d\frac{|x-y|^2}{t-s}\big)$.
 
 \end{enumerate}
 
\noindent For $\bar\alpha = \frac{\gamma-d}{2\gamma}$ with $\gamma\in(d, \infty)$, and any $\alpha<\bar \alpha$ there exists a constant $\bar K(\alpha)$ such that for all $0<s<t<T$, and all $x,y\in[0,1]$.

\hspace{3 pt}$(4)$ $\int_0^T\int_0^1|G_{t-\tau}-G_{s-\tau}|^2 d\eta d\tau \leq \bar K(\alpha)\rho((t,x), (s,y))^{2\alpha}$

\noindent where $\rho$ is the Euclidean distance in $[0,T]\times[0,1]$.

\section{Framework for the Uniform Laplace Principle}

In this section, we review some of the results presented in \cite{BDM}.  In particular, we state Theorem \ref{LDP_original} which asserts the uniform Laplace principle for a family of functionals of a Brownian sheet under two main assumptions. In subsection 3.1, we state the two assumptions for the class of semilinear SPDEs under study, and employ Theorem \ref{LDP_original} to show the uniform Laplace principle.

\subsection{Uniform Laplace Principle for Functionals of a Brownian Sheet.}

Let $(\Omega, {\mathcal F}, P, \{{\mathcal F}_t\})$ be the filtered probability space introduced as before, and $\psi: \Omega \times[0,T] \rightarrow L^2([0,1])$ an $L^2([0,1])$-valued predictable process.  Let ${\mathcal E}_0$ and ${\mathcal E}$ be Polish spaces, and let the initial condition $\eta$ take values in a compact subspace of ${\mathcal E}_0$.  Moreover, for every  
 $\varepsilon>0$, let  $\xi^{\varepsilon} : {\mathcal E}_0 \times {\mathcal C}([0,T]\times [0,1]; \mathbb R) \rightarrow {\mathcal E}$ be a family of measurable maps. Define $X^{\varepsilon, \eta} \doteq \xi^{\varepsilon} (\eta, \sqrt{\varepsilon }W)$, and introduce the following:

 \begin{equation}
 S^N \doteq \left\{\psi\in L^2([0,T]\times[0,1]): \int_{[0,T]\times[0,1]} \psi^2(s,y)ds dy \leq N\right\},  ~~~~ N\in \mathbb N, 
 \end{equation}
 
  \noindent $S^N$ is a compact metric space, equipped with the weak topology on $L^2([0,T]\times[0,1])$. For $v \in L^2([0,T]\times [0,1])$, define
 
 \begin{equation}
{\mathcal P}_2  \doteq \left\{ \psi : \int_0^T |\psi(s)|^2_{2} ds < \infty ~~a.s.\right\}.
\end{equation}
 \begin{equation}
{\mathcal P}^N_2 \doteq \left\{v\in   {\mathcal P}_2 : v(\omega) \in S^N, P-a.s. \right\}.
 \end{equation}
 
\noindent ${\mathcal P}_2$  is the space of controls.  By Girsonov's theorem the process
\begin{align*} 
  \widehat{W}(t) = W(t) + {\epsilon}^{-1/2} \int_0^t v(s) ds,
\end{align*}

\noindent is a cylindrical Wiener process under the measure $Q^{v,\epsilon}$ defined by

\begin{align*} \label{E:Girsanov change measure}
\frac{dQ^{v,\epsilon}}{dP}\doteq \exp\left\{-\frac{1}{\sqrt{\epsilon}}\int_0^T\int_0^1 v(s,y)W(dy ds) -\frac{1}{2\epsilon}\int_0^T\int_0^1v^2(s,y)dyds\right\}. 
\end{align*}

\noindent For convenience, define
 
 \begin{equation}
 {\mbox{Int}}(v)(t,x) \doteq \int_0^t\int_0^x v(s,y) ds dy.
 \end{equation}
 
\noindent The following condition is the standing assumption of Theorem \ref{LDP_original} which states the uniform Laplace principle for the family $\{X^{\varepsilon, \eta}\}.$
 
\noindent{\underline{CONDITION}}: There exists a measurable map $\xi^0: {\mathcal E}_0 \times {\mathcal C}([0,T]\times[0,1];\mathbb R) \rightarrow \mathbb {\mathcal E}$ such that
 
 \begin{enumerate}
 
 \item [(A1)] For every $M<\infty$ and compact set $K \subset {\mathcal E}_0 $, the set
 
 $$
 \Gamma_{M,K} \doteq \{\xi^0(\eta, \mbox{Int}(v)): v \in S^M, \eta\in K\},
 $$
 
 \noindent is a compact subset of ${\mathcal E}$.
 
\item [(A2)]  Consider $M<\infty$ and the families $\{v^{\epsilon}\} \subset {\mathcal P}^M_2$, and $\{\eta^{\epsilon}\} \subset {\mathcal E}_0$ such that $v^{\epsilon}\rightarrow v$, and $\eta^{\epsilon} \rightarrow \eta$ in distribution, as $\varepsilon \rightarrow 0$. Then

$$
\xi^{\varepsilon}\big(\eta^{\epsilon}, \sqrt{\varepsilon} W + {\mbox{Int}}(v^{\varepsilon})\big) \rightarrow \xi^0\big(\eta, \mbox{Int}(v)\big), 
$$
\noindent in distribution as $\epsilon \rightarrow 0$.
\end{enumerate}

\noindent For $h\in{ \mathcal E}$, and $\eta \in {\mathcal E}_0$, define the rate function

\begin{equation}
I_\eta(h) \doteq \inf_{\{v\in L^2([0,T]\times[0,1]) : h\doteq \xi^0(\eta, {Int}(v))\}}\left\{\frac{1}{2}\int_0^T\int_0^1 v^2(y,s) dy ds\right\}.
\end{equation}

\noindent The following theorem is due to Budhiraja et al. (\cite{BDM}, Theorem 7), and states the uniform Laplace principle for the family $\{X^{\varepsilon, \eta}\}$.

\begin{theorem}\label{LDP_original}Let $\xi^0: {\mathcal E}_0 \times {\mathcal C}([0,T]\times[0,1];\mathbb R) \rightarrow \mathbb {\mathcal E}$ be a measurable map satisfying conditions (A1) and (A2).  Suppose that for all $h\in{\mathcal E}$, $\eta\rightarrow I_{\eta}(h)$ is a lower semi-continuous map from ${\mathcal E}_0$ to $[0, \infty]$. Then for every $\eta \in {\mathcal E}_0 $, $I_{\eta}(h): {\mathcal E}\rightarrow [0,\infty]$, is a rate function on ${\mathcal E}$ and the family $\{I_{\eta}, \eta\in {\mathcal E}\}$ of rate functions has compact level sets on compacts.  Furthermore, the family $\{X^{\varepsilon, \eta}\}$ satisfies the the uniform Laplace principle on ${\mathcal E}$  with rate function $I_{\eta}$, uniformly in $\eta$ on compact subsets of ${\mathcal E}_0$.
\end{theorem} 

\subsection{The Controlled and Limiting Equations for the Semilinear SPDE} In the context of the semilinear SPDE under study, $\mathcal {E}_0 = L^{2}([0,1])$ is the space of the initial condition, and $\mathcal{E} = C([0,T];L^2([0,1]))$, the space of the solutions.  The solution map of Eq. (\ref{E:1}) is $U^\epsilon=\xi^\epsilon(\eta, \sqrt\epsilon W)$.  $V^{\epsilon,v}_{\eta} = \xi^\epsilon(\eta, \sqrt\epsilon W+ \mbox{Int}(v))$ is the solution map of the stochastic controlled equation for the semilinear SPDE,

\begin{align}
\label{controlled}
\frac{\partial V^{\epsilon}}{\partial t}(t,x) &=\nonumber \frac{\partial^2 V^{\epsilon}}{\partial x^2}(t,x) +\sqrt{\epsilon}\sigma(t, x, V^\epsilon(t,x))\frac{\partial^2W}{\partial t \partial x}(t,x)  +\frac{\partial}{\partial x}g(t,x,V^{\epsilon}(t,x)) \\&+ f(t,x,V^\epsilon(t,x)) + \sigma(t, x, V^\epsilon(t,x)) v(t,x),
\end{align}
whose mild solution is
\begin{align}
\label{E:controlled_process}
V^{\epsilon,v}_{\eta}(t,x)\nonumber &= \int_0^1 G_t(x,y)\eta(y)dy + \sqrt{\epsilon} \int_0^t\int_0^1 G_{t-s}(x,y) \sigma(s,V^{\epsilon,v}_{\eta} (s))(y)W(dy, ds)\\&\nonumber -\int_0^t\int_0^1\partial_yG_{t-s}(x,y)g(s, V^{\epsilon,v}_{\eta}(s))(y)dyds\\&\nonumber+\int_0^t\int_0^1G_{t-s}(x,y)f(s,V^{\epsilon,v}_{\eta} (s))(y)dyds\\&+\int_0^t\int_0^1G_{t-s}(x,y)\sigma(s,V^{\epsilon,v}_{\eta} (s))(y)v(s,y)dyds,
\end{align}

\noindent We refer to Eq. (\ref{E:controlled_process}) as the controlled process (i.e. the equation under the change of measure).  The map $V^{0,v}_{\eta} = \xi^0(\eta, \mbox{Int}(v)$ is the solution map of the limiting zero-noise equation, whose mild solution is  
\begin{align}
\label{E:limiting_process}
V^{0,v}_{\eta}(t,x)\nonumber &= \int_0^1 G_t(x,y)\eta(y)dy +\int_0^t\int_0^1G_{t-s}(x,y)\sigma(s,V^{0,v}_{\eta} (s))(y)v(s,y)dyds\\&\nonumber+\int_0^t\int_0^1G_{t-s}(x,y)f(s,V^{0,v}_{\eta} (s))(y)dyds\\&\nonumber-\int_0^t\int_0^1\partial_yG_{t-s}(x,y)g(s, V^{0,v}_{\eta}(s))(y)dyds.
\end{align}

\noindent We have the following existence and uniqueness result for the controlled process (\ref{E:controlled_process}), where the main ingredient of the proof is Girsonov's theorem.

\begin{theorem}[Existence and Uniqueness of Controlled Process]\label{existence_controlled} Let $\xi^{\epsilon}$ denote the solution mapping, and let $v\in {\mathcal P}^N_2$ for some $N\in {\mathbb N}$.  For $\epsilon>0$ and $\eta \in L^2([0,1])$ define
$$
V^{\epsilon, v}_{\xi} \doteq \xi^{\epsilon}\big(\eta, \sqrt{\epsilon}W+ \mbox{Int}(v)\big),
$$
then $V^{\epsilon, v}_{\eta}$ is the unique solution of equation (\ref{E:controlled_process}).

\end{theorem}

\begin{proof} For a fixed $v\in {\mathcal P}^N_2$, define 

$$
\frac{dQ^{v,\epsilon}}{dP}\doteq \exp\left\{-\frac{1}{\sqrt{\epsilon}}\int_0^T\int_0^1 v(s,y)W(dy ds) -\frac{1}{2\epsilon}\int_0^T\int_0^1v^2(s,y)dyds\right\}. 
$$

\noindent Since

$$
 \exp\left\{-\frac{1}{\sqrt{\epsilon}}\int_0^T\int_0^1 v(s,y)W(dy ds) -\frac{1}{2\epsilon}\int_0^T\int_0^1v^2(s,y)dyds\right\}, 
$$

\noindent is an exponential martingale, we have that $Q^{v,\epsilon}$ is a probability measure on $(\Omega, {\mathcal F}, P, \{{\mathcal F}_t\})$.  Obviously, $Q^{v,\epsilon}$ is equivalent to $P$.  By Girsanov's theorem (\cite{DZ}, Theorem 10.14 ), $\widehat W \doteq W+{\epsilon}^{-1/2}{\mbox{Int}}(u)$ is a Brownian sheet under $Q^{v,\epsilon}$.  By Theorem \ref{existence},  $V^{\epsilon, v}_{\xi}$ is the unique solution of (\ref{E:1}) with $\widehat W$ replaced by $W$ under the measure $Q^{v,\epsilon}$.  This is precisely equation (\ref{E:controlled_process}) on $(\Omega, {\mathcal F},Q^{\epsilon, v}, \{{\mathcal F}_t\})$.  By the equivalence of the measures, $V^{\epsilon, v}_{\eta}$ is the unique solution of Eq. (\ref{E:controlled_process}) under the measure $P$, and the proof is complete.
\end{proof}

\noindent For $h\in C\big([0,T]; L^2([0,1])\big)$, we define the following action functional

\begin{equation}\label{ratefunction}
I_\eta(h) \doteq \inf_v\int_0^T\int_0^1 v^2(s,y) dy ds,
\end{equation}

\noindent where the infimum is taken over all $v\in L^2([0,T]\times [0,1])$ such that

\begin{align}
\label{E:deterministic}
h&(t,x)\nonumber = \int_0^1 G_t(x,y)\eta(y)dy -\int_0^t\int_0^1\partial_yG_{t-s}(x,y)g(s, h(s))(y)dyds\\
&+\int_0^t\int_0^1G_{t-s}(x,y)f(s, h(s))(y)dyds+\int_0^t\int_0^1 G_{t-s}(x,y) \sigma(s,h(s)) v(s,y) dy ds.
\end{align}

\noindent The next Theorem asserts the existence and uniqueness of the limiting equation which we will use in the proof of Theorem \ref{convergence}.

\begin{theorem}[Existence and Uniqueness of Limiting Eqn]\label{uniqueness}Fix $\eta\in L^2([0,1])$ and $v \in L^2([0,T]\times[0,1])$.  Then there exists a unique function $h\in \mathcal C\big([0,T];L^2([0,1])\big)$ which satisfies equation (\ref{E:deterministic}).
\end{theorem}

\noindent The proof of this Theorem is very similar to that of Theorem \ref{existence}, and thus omitted.  We now state  two Theorems and two Lemmas which we are going to use in the proof of the main Theorem.  The next Lemma (\cite[Lemma 3.3]{G}) is used in proving the tightness of the second and third terms of the controlled process (\ref{E:controlled_process_perturbed}).

\begin{lemma}\label{Gyongy}
Let $\rho\in[1, \infty)$, and $q\in[1,\rho)$. Moreover, let $\zeta_n(t,y)$ be a sequence of random fields on $[0,T]\times[0,1]$ such that $\sup_{t\leq T}|\zeta_n(t,.)|_q \leq \theta_n$, where $\theta_n$ is a finite random variable for every $n$.  Assume that the sequence $\theta_n$ is bounded in probability, i.e.
$$
\lim_{c\rightarrow\infty}\sup_nP(\theta_n\geq C) =0.
$$
Then the sequence $J(\zeta_n) \doteq\int_0^t\int_0^1 R(r,t;x,y)\zeta_n(r,y)dy dr$, $t\in[0,T]$, $x\in[0,1]$ where $R(r,t;x,y) = \partial_y G(r,t;x,y)$ or $R(r,t;x,y) = G(r,t;x,y)$ is uniformly
 tight in $C\big([0,T]; L^\rho([0,1])\big)$.

\end{lemma}


\section{The Main Theorem}
We now announce the main theorem of this paper.

\begin{theorem}[Main Theorem]\label{LDP}

\noindent The processes $\{U^{\epsilon}(t): t\in[0,T]\}$ satisfy the uniform Laplace principle on $C\big([0,T]; L^2([0,1])\big)$ with rate function $I_{\eta}$ given by (\ref{ratefunction}).
\end{theorem}

\noindent In view of Theorem \ref{ratefunction}, it suffices to verify conditions (A1) and (A2).  Let $\beta:[0,1)\to[0,1)$ be a measurable map such that $\beta(r)\to\beta(0)=0$ as $r\to0$.

\subsection{Verification of Condition (A2)}  Condition (A2) follows by applying the following theorem with
$\beta(r)=r$, $r\in[0,1)$.

 \begin{theorem}[Convergence of the Controlled Process] \label{convergence}
 Let $M<\infty$, and suppose that $\eta^{\epsilon}\rightarrow\eta$ and $v^{\epsilon}\rightarrow v$ in distribution as $\epsilon\rightarrow 0$ with $\{v^{\epsilon}\} \subset {\mathcal P}_2^M$.  Then $V^{{\beta
 (\epsilon)}, v^{\epsilon}}_{\eta^{\epsilon}} \rightarrow V^{0,u}_{\eta}$ in distribution.
 \end{theorem}
 
\begin{proof}
 Note that 
\begin{align}
\label{E:controlled_process_perturbed}
V&^{\beta(\epsilon),v^\epsilon}_{\eta}(t,x)= \int_0^1 G_t(x,y)\eta^\epsilon(y)dy\nonumber\\&+ \sqrt{\beta(\epsilon)} \int_0^t\int_0^1 G_{t-s}(x,y) \sigma(s,V^{\beta(\epsilon),v^\epsilon}_{\eta^\epsilon} (s))(y)W(dy, ds)\nonumber\\&-\int_0^t\int_0^1\partial_yG_{t-s}(x,y)g(s, V^{\beta(\epsilon),v^\epsilon}_{\eta^\epsilon}(s))(y)dyds\nonumber\\&+\int_0^t\int_0^1G_{t-s}(x,y)f(s,V^{\beta(\epsilon),v^\epsilon}_{\eta^\epsilon} (s))(y)dyds\nonumber\\&+\int_0^t\int_0^1G_{t-s}(x,y)\sigma(s,V^{\beta(\epsilon),v^\epsilon}_{\eta^\epsilon} (s))(y)v^\epsilon(s,y)dyds\nonumber\\&\doteq J_1^{\epsilon}+J_2^{\epsilon}+J_3^{\epsilon}+J_4^{\epsilon}+J_5^{\epsilon}
\end{align}

\noindent We show tightness of  $J_i^{\epsilon}$  for $i = 1, 2, 3, 4, 5$ in $C\big([0,T]; L^2([0,1])\big)$, and therefore assert the claim.  Since $\eta^{\epsilon} \in L^2([0,1])$, the tightness of $J_1^{\epsilon}$ follows by the following lemma.

\begin{lemma}\label{continuity}  Let $\eta\in L^2([0,1])$.  Then $(t\rightarrow G_t\eta)$ belongs to $C\big([0,T]; L^2([0,1])\big)$, and 
\begin{align*}
\eta \rightarrow \{t\rightarrow  G_t\eta\},
\end{align*}

\noindent  is a continuous map in $\eta$.

\end{lemma}

\noindent  As for the tightness of $J_5^{\epsilon}$, we have
\begin{align}
\sup _{\epsilon\in (0,1)}J^{\epsilon}_5 \nonumber &\doteq \sup_{\epsilon\in(0,1)} \int_0^t\int_0^1 G_{t-s}(x,y)\sigma(s,V^{\beta(\epsilon),v^\epsilon}_{\eta^\epsilon} (s))(y)v^{\epsilon}(y,s) dyds \\&\leq  M\bigg(\int_0^t \int_0^1G_{t-s}^2(x,y) dyds \bigg)^{1/2}\sup_{\epsilon\in(0,1)}\bigg(\int_0^1\int_0^t (v^{\epsilon})^2 dy ds\bigg)^{1/2} \leq  C(T),
\end{align}

\noindent where H\"older's inequality, boundedness of $\sigma$, properties of the regularizing kernel, and boundedness of the controls in $L^2([0,T]\times[0,1])$ have been used.  This establishes the tightness of $J_5^{\epsilon}$.  As for the tightness of $J_4^{\epsilon}$, we mainly use Lemma \ref{Gyongy}. Note that $f$ satisfies the linear growth condition:

$$\sup_{t\in[0,T]}\sup_{x\in[0,1]}|f(t,x,r)|\leq K(1+|r|).$$\\
\noindent In Lemma \ref{Gyongy}, let $\rho=2$, $q=1$, and $\zeta^\epsilon(s,y) \doteq f(s,y, V^{\beta(\epsilon), v^\epsilon}_{\eta^\epsilon}(s,y))$.  We have

$$\sup_{t\in [0,T]}| f(s, V^{\beta(\epsilon), v^\epsilon}_{\eta^\epsilon}(s))|_1\leq K +K\sup_{t\in[0,T]} |V^{\epsilon, v^\epsilon}_{\eta^\epsilon}(s)|_2$$

\noindent  Let $\theta^\epsilon \doteq K +K\sup_{t\in[0,T]} |V^{\beta(\epsilon), v^\epsilon}_{\eta^\epsilon}(s)|_2$.  We have

\begin{align*}
\lim_{C\rightarrow\infty}\sup_\epsilon P(  K +K\sup_{t\in[0,T]} |V^{\beta(\epsilon), v^\epsilon}_{\eta^\epsilon}(s)|_2\geq C) & < \lim_{C\rightarrow\infty}\sup_\epsilon P(K\geq\frac{C}{2}) 
&\\+\lim_{C\rightarrow\infty}\sup_\epsilon P(\sup_{t\in[0,T]} |V^{\beta(\epsilon), v^\epsilon}_{\eta^\epsilon}(s)|_2\geq \frac{C}{2}) 
\end{align*}

\noindent Clearly the first term on the R.H.S. of the immediate above display in equal to zero.   As for the second term,  it suffices to show that  

$$
\sup_{t\leq T} |V^{\beta(\epsilon), v^{\epsilon}}_{\eta^{\epsilon}}(t,.)|_{2},
$$
 \noindent is bounded in probability, i.e.
 
\begin{equation}\label{bounded_probability}
 \lim_{C\rightarrow \infty}\sup_{{\epsilon}\in(0,1)}P\big( \sup_{t\leq T} |V^{\beta({\epsilon}), u^{\epsilon}}_{\xi^{\epsilon}}(t,.)|_{2} \geq C\big) = 0.
\end{equation}
 
\noindent The proof of (\ref{bounded_probability}) is similar to that in \cite{LS} but we include it here for the convenience of the reader.  Recall the class of stochastic semi-linear equations  (\ref{E:1}) which we rewrite here

\begin{align}
\label{approximation}
\frac{\partial u^{\epsilon}}{\partial t}(t,x) &=\nonumber \frac{\partial^2 u^{\epsilon}}{\partial x^2}(t,x) +\sqrt{\epsilon}\sigma(t, x, u^\epsilon(t,x))\frac{\partial^2W}{\partial t \partial x}(t,x) \\& +\frac{\partial}{\partial x}g(t,x,u^{\epsilon}(t,x))+ f(t,x,u^\epsilon(t,x)),
\end{align}

\noindent Note that the controlled equation (\ref{controlled}) can be recovered from the above equation.  In \cite{G}, Gy\"ongy (1998) proves the existence and uniqueness of the solutions to the above class of stochastic semi-linear equations, by an approximation procedure.  Let $f_n(t,x,r)$, and $ g_n(t,x,r)$ be sequences of bounded measurable functions such that they are globally Lipschitz in $r\in\mathbb R$, and $f_n \doteq f$, $g_n \doteq g$ for $|r|\leq n$, $f_n = g_n \doteq 0$ for $|r| \geq n+1$.  $f_n$, and $g_n$ satisfy the same growth conditions as $f$, and $g$. We have, by (\cite [Proposition 4.7] {G}), that there exists a unique solution, say $V^{\beta(\epsilon), v^{\epsilon}}_{\eta^{\epsilon},n}$, to the semi-linear equation (\ref{E:1}) with $f$ and $g$ replaced by $f_n$ and $g_n$.  That is,  $V^{\beta(\epsilon), v^{\epsilon}}_{\eta^{\epsilon},n}$, is the unique solution to the truncated equation.  Furthermore, $V^{\beta(\epsilon), v^{\epsilon}}_{\eta^{\epsilon},n}$ converges to $V^{\beta(\epsilon), v^{\epsilon}}_{\eta^{\epsilon}}$ in $C\big([0,T]; L^2([0,1])\big)$ in {{probability}}, and uniformly in $\epsilon$ as $n$ approaches infinity.  It has been demonstrated in \cite{G} that, for every $n\geq 1$

\begin{equation}\label{bounded_probability_approximation}
 \lim_{C\rightarrow \infty}\sup_{{\epsilon}\in(0,1)}P\big( \sup_{t\leq T} |V^{\beta({\epsilon}), v^{\epsilon}}_{{\eta^{\epsilon},n}}(t,.)|_{2} \geq C\big) = 0.
\end{equation}
 
\noindent Observe that

\begin{align}\label{ultimate}
\sup_{\epsilon \in (0,1)} &P \big(\sup_{t\leq T}|V^{\beta(\epsilon), v^{\epsilon}}_{\eta^{\epsilon}}|_{2}\geq C\big)\leq \sup_{\epsilon\in(0,1)}P\big(\sup_{t\leq T}|V^{\beta(\epsilon), v^{\epsilon}}_{\eta^{\epsilon}} -V^{\beta({\epsilon}), v^{\epsilon}}_{{\eta^{\epsilon},n}}|_{2}\nonumber\\& +\sup_{t\leq T}|V^{\beta({\epsilon}), v^{\epsilon}}_{{\eta^{\epsilon},n}}   |_{2} \geq C\big)\nonumber \nonumber\\& \leq \sup_{\epsilon\in(0,1)}P\big(\sup_{t\leq T}|V^{\beta(\epsilon), v^{\epsilon}}_{\eta^{\epsilon}} -V^{\beta({\epsilon}), v^{\epsilon}}_{{\eta^{\epsilon},n}}|_{2} \geq \frac{C}{2}\big) \\&+\sup_{\epsilon\in(0,1)} P\big(\sup_{t\leq T}|V^{\beta({\epsilon}), v^{\epsilon}}_{{\eta^{\epsilon},n}}   |_{2}\geq \frac{C}{2}\big).
\end{align}

\noindent By letting $C$ approach infinity, and exploiting the boundedness in probability of  $|V^{\beta({\epsilon}), v^{\epsilon}}_{\eta^{\epsilon},n}|_{2}$, we get

\begin{align*}
\lim_{C\rightarrow\infty}\sup_{\epsilon \in (0,1)}& P \big(\sup_{t\leq T}|V^{\beta(\epsilon), v^{\epsilon}}_{\eta^{\epsilon}}|_{2} \geq C\big)\\&\leq \lim_{C\rightarrow \infty}\sup_{\epsilon\in(0,1)}P\big(\sup_{t\leq T}|V^{\beta(\epsilon), v^{\epsilon}}_{\eta^{\epsilon}} -V^{\beta({\epsilon}), v^{\epsilon}}_{{\eta^{\epsilon},n}}|_{2} \geq \frac{C}{2}\big). 
\end{align*}

\noindent Now by letting $n$ tend to infinity, due the convergence in probability of $V^{\beta(\epsilon), v^{\epsilon}}_{\eta^{\epsilon},n}$  to $V^{\beta(\epsilon), v^{\epsilon}}_{\eta^{\epsilon}}$, we conclude that

\begin{equation}\label{tightness}
 \lim_{C\rightarrow \infty}\sup_{{\epsilon}\in(0,1)}P\big( \sup_{t\leq T} |V^{\beta({\epsilon}), u^{\epsilon}}_{\xi^{\epsilon}}(t,.)|_{2} \geq C\big) = 0.
\end{equation}
\noindent Therefore

$$\lim_{C\rightarrow\infty}\sup_\epsilon P(\theta ^\epsilon \geq C) = 0,$$

\noindent and the assumption of Lemma  \ref{Gyongy} is satisfied.  This establishes the tightness of $J_4^{\epsilon}$.  The proof of tightness for $J_3^{\epsilon}$ follows by the same analogy as $J_5^{\epsilon}$, and thus omitted.   Therefore, the tightness of $V^{\beta({\epsilon}), u^{\epsilon}}_{\xi^{\epsilon}}$ in  $C\big([0,T]; L^2([0,1])\big)$ is concluded.

Having the tightness of  $J_i^{\epsilon}$ for $i = 1, 2, 3, 4, 5$ at hand, by Prohorov's theorem, we can extract a subsequence along which each of the aforementioned processes and $V^{\beta({\epsilon}), v^{\epsilon}}_{\eta^{\epsilon}}$ converge in distribution to $J_i^0$ and $V^{0,v}_{\eta}(t,x)$ in $C\big([0,T]; L^2([0,1])\big)$.  We aim to show that the respective limits are as follows:

\noindent \begin{align}
&\nonumber J_1^0  = \int_0^1G_t(s,y)\xi(y) dy,\\&\nonumber J_2^0 = 0,
\\&\nonumber
J_3^0 = -\int_0^t\int_0^1 \partial_yG_{t-s}(x,y) g(s,V^{{0}, v}_{\eta}(s))(y)  dy ds,
\\&\nonumber 
J_4^0 = \int_0^t\int_0^1 G_{t-s}(x,y) f(s,V^{{0}, v}_{\eta}(s))(y)  dy ds,
\\&\nonumber 
J_5^0 = \int_0^t\int_0^1 G_{t-s}(x,y) \sigma(s,V^{{0}, v}_{\eta}(s))(y) v(s,y) dy ds.
\end{align}

\noindent The case $i=1$ follows from lemma (\ref{continuity}).  The case $i=2$ follows from Lemma 3 in \cite{BDM}.  Note that convergence in probability in $C([0,T] \times [0,1])$ implies the same in $C\big([0,T]; L^2 ([0,1])\big)$.  As for $i=3$, we invoke the Skorokhod  Representation Theorem \cite{DK}, and thus assume almost sure convergence on a larger, common probability space. Denote the RHS of $J_3^0$ by $\bar J_3^0$.  We have

\begin{align}\label{g_convergence}
&|J_3^{\epsilon} - \bar J_3^0| \nonumber \leq \int_0^t\int_0^1 |\partial_y G_{t-s}|(|1+|V^{\beta(\epsilon), v^\epsilon}_{{\eta}^\epsilon}|+|V^{0, v}_{\eta}|) |V^{\beta(\epsilon), v^\epsilon}_{{\eta}^\epsilon}-V^{0, v}_{\eta}|dy ds \\&\nonumber\leq (\sup_{x,t}|V^{\beta(\epsilon), v^\epsilon}_{{\eta}^\epsilon}-V^{0, v}_{\eta}| ) \bigg(T (\sup_t  |V^{\beta(\epsilon), v^\epsilon}_{{\eta}^\epsilon}|_2 +\sup_t|V^{{0}, v}_{\eta}|_2)\big(\int_0^t\int_0^1 |\partial_y G_{t-s}|^2\big)^{1/2}
\\&+\int_0^t\int_0^1 |\partial_y G_{t-s}|dy ds\bigg)  ,
\end{align}

\noindent where the Lipschitz property of $g$  with linearly growing constant, H\"older's inequality, and the properties of the regularizing kernel have been used. The right-hand-side of (\ref{g_convergence}) thus converges to zero as $\epsilon \rightarrow 0$ since $V^{\beta(\epsilon), v^\epsilon}_{{\eta}^\epsilon}\to V^{0, v}_{\eta}$, and 


$$
\int_0^t\int_0^1 |\partial_yG_{t-s}|^2 dy ds \leq C(T),
$$
\noindent By the fact that the limit is unique, and that $\bar J_3^0$ is a continuous random field (by Theorem \ref{uniqueness}) we conclude that $J_3^{0} = \bar J_3^0$.  The case $i=4$ follows by the same exact analogy as the third case.  For $i=5$, we invoke the Skorokhod  Representation Theorem \cite{DK} again.  Denote the right-hand-side of $J_5^0$ by $\bar J_5^0$.  We have
\begin{align}\label{Eq:control_convergence}
|J_5^{\epsilon} - \bar J_5^0| &\nonumber\leq \int_0^t\int_0^1 |G_{t-s}| |\sigma(s,V^{\beta({\epsilon}), v^\epsilon}_{{\eta}^\epsilon}(s))(y)-\sigma(s,V^{{0}, v}_{\eta}(s))(y)| |v^\epsilon(s,y)|\\& +\int_0^t\int_0^1 |G_{t-s}| \sigma(s,V^{{0}, v}_{\eta}(s))(y) |v^\epsilon(s,y) - v(s,y)|dsdy
\end{align} 
\noindent The first term on the right-hand-side of (\ref{Eq:control_convergence}) can be bounded above by
\begin{align} \label{cauchy-schwartz}
M &\nonumber \big[ \int_0^t\int_0^1 |G_{t-s}|^2 |\sigma(s,V^{\beta({\epsilon}), v^\epsilon}_{\eta^\epsilon}(s))(y)-\sigma(s,V^{{0}, v}_{\eta}(s))(y)|^2 dy ds \big]^{1/2}\\& \leq C(T) (\sup_{x,t}| V^{\beta({\epsilon}), v^\epsilon}_{{\eta}^\epsilon} -V^{{0}, v}_{\eta}|)
\end{align}

\noindent where the Cauchy-Schwartz inequality, the properties of the regularizing kernel and control, and the Lipschitz property of $\sigma$ have been used.  The first term on the RHS of (\ref{Eq:control_convergence}) thus converges to zero, since $V^{\beta({\epsilon}), v^\epsilon}_{{\eta}^\epsilon}\to V^{{0}, v}_{\eta}$ as $\epsilon \to 0$ .  The second term on the RHS of (\ref{Eq:control_convergence}) also converges to zero  as $\epsilon \rightarrow 0$, since ${v^{{\epsilon}}\to v}$, and 

$$
\int_0^t\int_0^1 |G_{t-s}|  \sigma(s,V^{{0}, v}_{\eta}(s))(y)dy ds <\infty.
$$

\noindent   Again, by the fact that the limit is unique, and that $\bar J_5^0$ is a continuous random field (by Theorem \ref{uniqueness}) we conclude that $J_5^{0} = \bar J_5^0$. Thus, we have proven that along a subsequence, the controlled process converges to the limiting equation. 
 \end{proof}

 \subsection{Verification of Condition (A1)}  Condition (A1) follows by Theorem \ref{uniqueness}, and applying Theorem \ref{convergence} with $\beta=0$.  This concludes the proof of Theorem \ref{LDP}.


\bibliographystyle{amsplain}

\end{document}